\documentclass[12pt]{article}
\usepackage{amsmath,amssymb,amsbsy,amsfonts,amsthm,latexsym,
            amsopn,amstext,amsxtra,euscript,amscd}
\begin{document}
\bibliographystyle{plain}
\newfont{\teneufm}{eufm10}
\newfont{\seveneufm}{eufm7}
\newfont{\fiveeufm}{eufm5}
%
%
\newfam\eufmfam
              \textfont\eufmfam=\teneufm \scriptfont\eufmfam=\seveneufm
              \scriptscriptfont\eufmfam=\fiveeufm
\def\bbbr{{\rm I\!R}}
\def\bbbm{{\rm I\!M}}
\def\bbbn{{\rm I\!N}}
\def\bbbf{{\rm I\!F}}
\def\bbbh{{\rm I\!H}}
\def\bbbk{{\rm I\!K}}
\def\bbbp{{\rm I\!P}}
\def\bbbone{{\mathchoice {\rm 1\mskip-4mu l} {\rm 1\mskip-4mu l}
{\rm 1\mskip-4.5mu l} {\rm 1\mskip-5mu l}}}
\def\bbbc{{\mathchoice {\setbox0=\hbox{$\displaystyle\rm C$}\hbox{\hbox
to0pt{\kern0.4\wd0\vrule height0.9\ht0\hss}\box0}}
{\setbox0=\hbox{$\textstyle\rm C$}\hbox{\hbox
to0pt{\kern0.4\wd0\vrule height0.9\ht0\hss}\box0}}
{\setbox0=\hbox{$\scriptstyle\rm C$}\hbox{\hbox
to0pt{\kern0.4\wd0\vrule height0.9\ht0\hss}\box0}}
{\setbox0=\hbox{$\scriptscriptstyle\rm C$}\hbox{\hbox
to0pt{\kern0.4\wd0\vrule height0.9\ht0\hss}\box0}}}}
\def\bbbq{{\mathchoice {\setbox0=\hbox{$\displaystyle\rm
Q$}\hbox{\raise
0.15\ht0\hbox to0pt{\kern0.4\wd0\vrule height0.8\ht0\hss}\box0}}
{\setbox0=\hbox{$\textstyle\rm Q$}\hbox{\raise
0.15\ht0\hbox to0pt{\kern0.4\wd0\vrule height0.8\ht0\hss}\box0}}
{\setbox0=\hbox{$\scriptstyle\rm Q$}\hbox{\raise
0.15\ht0\hbox to0pt{\kern0.4\wd0\vrule height0.7\ht0\hss}\box0}}
{\setbox0=\hbox{$\scriptscriptstyle\rm Q$}\hbox{\raise
0.15\ht0\hbox to0pt{\kern0.4\wd0\vrule height0.7\ht0\hss}\box0}}}}
\def\bbbt{{\mathchoice {\setbox0=\hbox{$\displaystyle\rm
T$}\hbox{\hbox to0pt{\kern0.3\wd0\vrule height0.9\ht0\hss}\box0}}
{\setbox0=\hbox{$\textstyle\rm T$}\hbox{\hbox
to0pt{\kern0.3\wd0\vrule height0.9\ht0\hss}\box0}}
{\setbox0=\hbox{$\scriptstyle\rm T$}\hbox{\hbox
to0pt{\kern0.3\wd0\vrule height0.9\ht0\hss}\box0}}
{\setbox0=\hbox{$\scriptscriptstyle\rm T$}\hbox{\hbox
to0pt{\kern0.3\wd0\vrule height0.9\ht0\hss}\box0}}}}
\def\bbbs{{\mathchoice
{\setbox0=\hbox{$\displaystyle     \rm S$}\hbox{\raise0.5\ht0\hbox
to0pt{\kern0.35\wd0\vrule height0.45\ht0\hss}\hbox
to0pt{\kern0.55\wd0\vrule height0.5\ht0\hss}\box0}}
{\setbox0=\hbox{$\textstyle        \rm S$}\hbox{\raise0.5\ht0\hbox
to0pt{\kern0.35\wd0\vrule height0.45\ht0\hss}\hbox
to0pt{\kern0.55\wd0\vrule height0.5\ht0\hss}\box0}}
{\setbox0=\hbox{$\scriptstyle      \rm S$}\hbox{\raise0.5\ht0\hbox
to0pt{\kern0.35\wd0\vrule height0.45\ht0\hss}\raise0.05\ht0\hbox
to0pt{\kern0.5\wd0\vrule height0.45\ht0\hss}\box0}}
{\setbox0=\hbox{$\scriptscriptstyle\rm S$}\hbox{\raise0.5\ht0\hbox
to0pt{\kern0.4\wd0\vrule height0.45\ht0\hss}\raise0.05\ht0\hbox
to0pt{\kern0.55\wd0\vrule height0.45\ht0\hss}\box0}}}}
\def\bbbz{{\mathchoice {\hbox{$\sf\textstyle Z\kern-0.4em Z$}}
{\hbox{$\sf\textstyle Z\kern-0.4em Z$}}
{\hbox{$\sf\scriptstyle Z\kern-0.3em Z$}}
{\hbox{$\sf\scriptscriptstyle Z\kern-0.2em Z$}}}}
\def\ts{\thinspace}

\newtheorem{theorem}{Theorem}
\newtheorem{lemma}[theorem]{Lemma}
\newtheorem{claim}[theorem]{Claim}
\newtheorem{cor}[theorem]{Corollary}
\newtheorem{prop}[theorem]{Proposition}
\newtheorem{definition}[theorem]{Definition}
\newtheorem{remark}[theorem]{Remark}
\newtheorem{question}[theorem]{Open Question}

\def\qed{\ifmmode
\squareforqed\else{\unskip\nobreak\hfil
\penalty50\hskip1em\null\nobreak\hfil\squareforqed
\parfillskip=0pt\finalhyphendemerits=0\endgraf}\fi}

\def\squareforqed{\hbox{\rlap{$\sqcap$}$\sqcup$}}

\def \C {{\mathbb C}}
\def \F {{\mathbb F}}
\def \L {{\mathbb L}}
\def \K {{\mathbb K}}
\def \Q {{\mathbb Q}}
\def \Z {{\mathbb Z}}
\def\cA{{\mathcal A}}
\def\cB{{\mathcal B}}
\def\cC{{\mathcal C}}
\def\cD{{\mathcal D}}
\def\cE{{\mathcal E}}
\def\cF{{\mathcal F}}
\def\cG{{\mathcal G}}
\def\cH{{\mathcal H}}
\def\cI{{\mathcal I}}
\def\cJ{{\mathcal J}}
\def\cK{{\mathcal K}}
\def\cL{{\mathcal L}}
\def\cM{{\mathcal M}}
\def\cN{{\mathcal N}}
\def\cO{{\mathcal O}}
\def\cP{{\mathcal P}}
\def\cQ{{\mathcal Q}}
\def\cR{{\mathcal R}}
\def\cS{{\mathcal S}}
\def\cT{{\mathcal T}}
\def\cU{{\mathcal U}}
\def\cV{{\mathcal V}}
\def\cW{{\mathcal W}}
\def\cX{{\mathcal X}}
\def\cY{{\mathcal Y}}
\def\cZ{{\mathcal Z}}
\newcommand{\rmod}[1]{\: \mbox{mod}\: #1}

\def\tcN{\cN^\mathbf{c}}
\def\F{\mathbb F}
\def\Tr{\operatorname{Tr}}
\def\mand{\qquad \mbox{and} \qquad}
\renewcommand{\vec}[1]{\mathbf{#1}}
\def\eqref#1{(\ref{#1})}
\newcommand{\ignore}[1]{}
\hyphenation{re-pub-lished}
\parskip 1.5 mm
\def\lln{{\mathrm Lnln}}
\def\Res{\mathrm{Res}\,}
\def\F{{\bbbf}}
\def\Fp{\F_p}
\def\fp{\Fp^*}
\def\Fq{\F_q}
\def\ff{\F_2}
\def\ffn{\F_{2^n}}
\def\K{{\bbbk}}
\def \Z{{\bbbz}}
\def \N{{\bbbn}}
\def\Q{{\bbbq}}
\def \R{{\bbbr}}
\def \P{{\bbbp}}
\def\Zm{\Z_m}
\def \Um{{\mathcal U}_m}
\def \Bf{\frak B}
\def\Km{\cK_\mu}
\def\va {{\mathbf a}}
\def \vb {{\mathbf b}}
\def \vc {{\mathbf c}}
\def\vx{{\mathbf x}}
\def \vr {{\mathbf r}}
\def \vv {{\mathbf v}}
\def\vu{{\mathbf u}}
\def \vw{{\mathbf w}}
\def \vz {{\mathbfz}}
\def\\{\cr}
\def\({\left(}
\def\){\right)}
\def\fl#1{\left\lfloor#1\right\rfloor}
\def\rf#1{\left\lceil#1\right\rceil}
\def\flq#1{{\left\lfloor#1\right\rfloor}_q}
\def\flp#1{{\left\lfloor#1\right\rfloor}_p}
\def\flm#1{{\left\lfloor#1\right\rfloor}_m}
\def\Al{{\sl Alice}}
\def\Bob{{\sl Bob}}
\def\Or{{\mathcal O}}
\def\inv#1{\mbox{\rm{inv}}\,#1}
\def\invM#1{\mbox{\rm{inv}}_M\,#1}
\def\invp#1{\mbox{\rm{inv}}_p\,#1}
\def\Ln#1{\mbox{\rm{Ln}}\,#1}
\def \nd {\,|\hspace{-1.2mm}/\,}
\def\ord{\mu}
\def\E{\mathbf{E}}
\def\Cl{{\mathrm {Cl}}}
\def\epp{\mbox{\bf{e}}_{p-1}}
\def\ep{\mbox{\bf{e}}_p}
\def\eq{\mbox{\bf{e}}_q}
\def\bm{\bf{m}}
\newcommand{\floor}[1]{\lfloor {#1} \rfloor}
\newcommand{\comm}[1]{\marginpar{
\vskip-\baselineskip
\raggedright\footnotesize
\itshape\hrule\smallskip#1\par\smallskip\hrule}}
\def\rem{{\mathrm{\,rem\,}}}
\def\dist {{\mathrm{\,dist\,}}}
\def\etal{{\it et al.}}
\def\ie{{\it i.e. }}
\def\veps{{\varepsilon}}
\def\eps{{\eta}}
\def\ind#1{{\mathrm {ind}}\,#1}
               \def \MSB{{\mathrm{MSB}}}
\newcommand{\abs}[1]{\left| #1 \right|}

\title {On some conjectures on the monotonicity of some combinatorial sequences}
%

\author{{\sc Florian~Luca}\\
Centro de Ciencias Matem{\'a}ticas,\\
Universidad Nacional Autonoma de M{\'e}xico,\\
C.P. 58089, Morelia, Michoac{\'a}n, M{\'e}xico\\
{\tt fluca@matmor.unam.mx}
\and
{\sc Pantelimon St\u anic\u a}  \\
Naval Postgraduate School \\
Applied Mathematics Department \\
Monterey, CA 93943, USA\\
{\tt pstanica@nps.edu}
}

\date{\today}
\pagenumbering{arabic}

\maketitle

\begin{abstract} Here, we prove some conjectures on the monotony of combinatorial sequences from the recent preprint of Zhi--Wei Sun~\cite{Sun}.
\end{abstract}

\section{Introduction}

In \cite{Sun}, it was conjectured that the sequences of general term $a_{n+1}^{1/(n+1)}/a_n^{1/n}$, or $a_n^{1/n}$ are monotonically increasing (or decreasing) for all $n\ge n_0$, for a large class of sequences ${\bf a}=(a_n)_{n\ge 1}$ appearing in combinatorics. Two of these conjectures were confirmed recently in~\cite{HSW12}. Here, we confirm eight more of these conjectures (some partially, up to an explicit starting index $n_0$).

For a sequence ${\bf u}=(u_n)_{n\ge 1}$ and a positive integer $k$ we write
$$
\Delta^{(k)} u_n=u_{n+k}-\binom{k}{1} u_{n+k-1}+\cdots+(-1)^j \binom{k}{j} u_{n+k-j}+\cdots+(-1)^k u_n,
$$
for the $k$th iterated difference of $(u_{n+k},\ldots,u_n)$.

\section{Bernoulli, Tangent and Euler numbers}

The Bernoulli numbers $B_0, B_1, B_2, \ldots$ are rational numbers given by
$$
B_0=1\quad {\text{\rm and}}\quad \sum_{k=0}^n \binom{n+1}{k} B_k=0\quad {\text{\rm for~all}}\quad n\ge 1,
$$
whose exponential generating function is
 \[
 \frac{z}{e^z-1}=\sum_{k=0}^\infty \frac{B_n z^n}{n!}.
 \]
It is well-known that $B_{2n+1}=0$ for all $n\ge 1$.
Closely connected to the Bernoulli numbers are the Tangent numbers $T_n$ and the Euler numbers $E_n$, defined
by their exponential generating functions
\begin{eqnarray}
\tan z&=&\sum_{k=0}^\infty (-1)^{k+1} \frac{T_{2k+1} z^{2k+1}}{(2k+1)!},\\
\sec z&=&\sum_{k=0}^\infty (-1)^k \frac{E_{2k} z^{2k}}{(2k)!}.
\end{eqnarray}
Thus, $E_{2k-1}=0, T_{2k}=0$, $k\geq 1$.
We recall Stirling's formula
\begin{equation}
\label{eq:Stirling}
n!=(n/e)^n {\sqrt{2\pi n}}\, e^{\theta_n},\quad {\text{\rm where}}\quad \frac{1}{12n+1}<\theta_n<\frac{1}{12n},\quad {\text{\rm for~all}}\quad n\ge 1.
\end{equation}

Our first result gives an affirmative answer to Conjecture 2.8 and 3.5 in \cite{Sun}.

\begin{sloppypar}
\begin{theorem}
\label{thm:Bern}
The sequence $(|B_{2n}|^{1/n})_{n\ge 1}$, $(|T_{2n-1}|^{1/n})_{n\ge 1}$, $(|E_{2n}|^{1/n})_{n\ge 1}$ are all increasing. Furthermore, the sequences $(|B_{2n+2}|^{1/(n+1)}/|B_{2n}|^{1/n})_{n\ge 2}$,
$(|T_{2n+1}|^{1/(n+1)}/|T_{2n-1}|^{1/n})_{n\ge 1}$, $(|E_{2n+2}|^{1/(n+1)}/|E_{2n}|^{1/n})_{n\ge 1}$ are decreasing.
\end{theorem}
\end{sloppypar}

\begin{proof}
We start with the Bernoulli numbers. We use the formula
$$
|B_{2n}|=\frac{2(2n!)}{(2\pi)^{2n}}\zeta(2n).
$$
Clearly,
$$
\zeta(2n)=1+\frac{1}{2^{2n}}+\frac{1}{3^{2n}}+\cdots=1+\eta_n,
$$
where
\begin{equation}
\begin{split}
\label{eq:etan}
\eta_n&=\frac{1}{2^{2n}}\left(1+\frac{1}{1.5^{2n}}+\frac{1}{2^{2n}}+\cdots\right)\\
&\le \frac{1}{2^{2n}} \left(1+\frac{1}{1.5^{2n}}+2(\zeta(2n)-1)\right)\le
\frac{3}{2^{2n}}
\end{split}
\end{equation}
for $n\ge 1$.
Thus, putting $|B_{2n}|=\exp v_n$, we have that
\begin{eqnarray}
\label{eq:10}
\frac{v_n}{n} & = & \frac{\log\left(2 (2n)! (2\pi)^{-2n} (1+\eta_n)\right)}{n}\nonumber\\
& = &
\frac{\log 2}{n} +\frac{\log (2n)!}{n}-2\log (2\pi) +\frac{\log(1+\eta_n)}{n}\nonumber\\
& = &
\frac{\log 2}{n} + 2\log(2n)-2+\frac{\log(4\pi n)}{2n}+\frac{\theta_{2n}}{n} 
-2\log (2\pi) +\frac{\log(1+\eta_n)}{n} \nonumber\\
& = &
2\log n+c+\frac{\log n}{2n}+\frac{\log(16\pi)}{2n}+\frac{\theta_{2n}}{n}+\frac{\log (1+\eta_n)}{n},
\end{eqnarray}
where $c=2\log 2-2-2\log (2\pi)=-2-2\log \pi$.
Taking the first iterated difference in \eqref{eq:10} above, we get
\begin{eqnarray*}
\Delta^{(1)} \left(\frac{v_n}{n}\right) & = & \frac{v_{n+1}}{n+1}-\frac{v_n}{n}\\
& = & 2\log\left(1+\frac{1}{n}\right)+\left(\frac{\log(n+1)}{2(n+1)}-\frac{\log n}{2n}\right)-\frac{\log(16 \pi)}{2n(n+1)} \\
& + & \left(\frac{\theta_{2n+2}}{n+1}-\frac{\theta_{2n}}{n}\right)+
\left(\frac{\log(1+\eta_{n+1})}{n+1}-\frac{\log(1+\eta_n)}{n}\right).
\end{eqnarray*}
By the Intermediate Value Theorem
$$
\frac{\log(n+1)}{n+1}-\frac{\log n}{n}= \frac{d}{dx} \left(\frac{\log x}{x} \right)\Big|_{x=\zeta\in [n,n+1)},
$$
therefore
$$
\left|\frac{\log(n+1)}{n+1}-\frac{\log n}{n}\right|< \frac{\log(n+1)}{n^2}.
$$
Using the inequalities
$$
\log\left(1+\frac{1}{n}\right)\ge \frac{1}{2n}\quad {\text{\rm for}}\quad n\geq 1,
$$
\begin{equation}
\label{eq:logg}
\log(1+x)\le x\quad {\text{\rm for~all~real~numbers}}\quad x,
\end{equation}
with $x=\eta_n$ and $x=\eta_{n+1}$, inequality \eqref{eq:etan} together with  Stirling's formula~\eqref{eq:Stirling},
we get
\begin{equation}
\label{eq:Delta1Ber}
\Delta^{(1)} \left(\frac{v_n}{n}\right)\ge \frac{1}{n}-\frac{\log(n+1)}{2n^2}-\frac{\log(16\pi)}{2n(n+1)}-\frac{1}{12n^2}-\frac{6}{2^{2n}n},
\end{equation}
and the above expression is positive for $n\ge 3$. This proves that $|B_{2n}|^{1/n}$ is increasing for $n\ge 3$, and by hand one checks that this is in fact true for all $n\ge 1$.

Taking now the second iterated difference in \eqref{eq:10}, one gets
\begin{eqnarray}
\label{eq:Delta2Ber}
\Delta^{(2)} \left(\frac{v_n}{n}\right) & = & \left(\frac{v_{n+2}}{n+2}-\frac{v_{n+1}}{n+1}\right)-\left(\frac{v_{n+1}}{n+1}-\frac{v_n}{n}\right)\nonumber\\
& = & 2\left(\log\left(1+\frac{1}{n+1}\right)-\log\left(1+\frac{1}{n}\right)\right)\nonumber\\
& + & \frac{1}{2}\left(\frac{\log(n+2)}{(n+2)}-\frac{2\log(n+1)}{(n+1)}+
\frac{\log n}{n}\right)\nonumber\\
& + & \frac{\log(16\pi)}{n(n+1)(n+2)}\nonumber\\
& + & \left(\frac{\theta_{2n+4}}{n+2}-\frac{2\theta_{2n+2}}{n+1}+\frac{\theta_{2n}}{n}\right)\nonumber\\
& + &
\left(\frac{\log(1+\eta_{n+2})}{n+2}-\frac{2\log(1+\eta_{n+1})}{n+1}+\frac{\log(1+\eta_n)}{n}\right).
\end{eqnarray}
We have
\begin{equation}
\label{eq:4}
\log\left(1+\frac{1}{n+1}\right)-\log\left(1+\frac{1}{n}\right)=\log\left(1-\frac{1}{(n+1)^2}\right)< -\frac{1}{(n+1)^2}.
\end{equation}
\begin{eqnarray}
\label{eq:5}
\frac{\log(n+2)}{n+2} & - & \frac{2\log(n+1)}{n+1}+
\frac{\log n}{n}  = (\log n)\left(\frac{1}{n+2}-\frac{2}{n+1}+\frac{1}{n}\right)\nonumber\\
& + & \frac{1}{n+2} \log\left(1+\frac{2}{n}\right)-\frac{2}{n+1}\log\left(1+\frac{1}{n}\right)\nonumber\\
& \le &
\frac{2\log n}{n(n+1)(n+2)}
+  \frac{2}{n(n+2)}-\frac{2}{n(n+1)}+\frac{2}{(n+1)n^2}\nonumber\\
& = & \frac{2\log n+4}{n(n+1)(n+2)},
\end{eqnarray}
where we have used the fact that
$$
x-x^2<\log(1+x)<x\quad {\text{\rm holds~for~ all}}\quad x\in (0,1/2).
$$
Next,
\begin{equation}
\label{eq:6}
\frac{\theta_{2n+4}}{n+2}-\frac{2\theta_{2n+2}}{n+1}+\frac{\theta_{2n}}{n}<\frac{1}{6n^2},
\end{equation}
by \eqref{eq:Stirling}. Further, by using inequality \eqref{eq:logg} with $x=\eta_n,~x=\eta_{n+1},~ x=\eta_{n+2}$ together with inequality 
\eqref{eq:etan}, we get
\begin{equation}
\label{eq:7}
\left|\frac{\log(1+\eta_{n+2})}{n+2}-\frac{2\log(1+\eta_{n+1})}{n+1}+\frac{\log(1+\eta_n)}{n}\right|<\frac{12}{2^{2n}n}
\end{equation}
for $n\ge 3$.
Putting all these together, we have
\begin{equation}
\label{eq:Delta2}
\Delta^{(2)} \left(\frac{v_n}{n}\right)\le -\frac{2}{(n+1)^2}+\frac{\log n+2+2\log(16\pi)}{2n(n+1)(n+2)}+\frac{1}{6 n^2}+\frac{12}{n 2^{2n}},
\end{equation}
and this last expression is negative for $n\ge 4$. So, the sequence of general term $|B_{2n+2}|^{1/(n+1)}/|B_{2n}|^{1/n}$ is increasing for $n\ge 4$, and  then one checks by hand that it is also increasing for $n=2,3$.

We next deal with the Tangent numbers. We have (see \cite{BBD}),
\begin{equation}
\label{eq:Tan1}
|T_{2n-1}|=2^{2n}(2^{2n}-1) \frac{|B_{2n}|}{2n}=4^{2n} \left(\frac{2(2n)!}{(2\pi)^{2n}}\right) \left(\frac{1}{2n}\right) \left(\left(1-\frac{1}{2^{2n}}\right)\zeta(2n)\right).
\end{equation}
Since
$$
1<\left(1-\frac{1}{2^{2n}}\right)\zeta(2n)<\zeta(2n),
$$
it follows by \eqref{eq:etan} that
\begin{equation}
\label{eq:etan1}
\left(1-\frac{1}{2^{2n}}\right)\zeta(2n)=1+\eta_n\qquad {\text{\rm for~some}}\qquad 0<\eta_n<\frac{3}{2^{2n}}.
\end{equation}
Writing $|T_{2n-1}|=\exp v_n$ and following along calculation \eqref{eq:10}, we get that
\begin{eqnarray}
\label{eq:Tan}
\frac{v_n}{n} & = & 4\log 2+\frac{1}{n} \log\left(\frac{2(2n)!}{(2\pi)^{2n}}\right) -\frac{\log(2n)}{n}+\frac{\log(1+\eta_n)}{n}\nonumber\\
& = & 2\log n+c_1-\frac{\log n}{2n}+\frac{\log(4\pi)}{2n}+\frac{\theta_{2n}}{n}+\frac{\log (1+\eta_n)}{n},
\end{eqnarray}
where $c_1=4\log 2+c=4\log 2-2-2\log \pi$. Comparing the last row of \eqref{eq:10} with the last row of \eqref{eq:Tan}, we see that the only differences are in the value of $c$, the fact that the term $(\log n)/(2n)$ has now changed sign and the positive constant $\log(16 \pi)$ has been replaced by the smaller positive constant $\log(4\pi)$. Following along the arguments from \eqref{eq:Delta1Ber} and \eqref{eq:Delta2Ber}, we note that such changes do not induce any significant change in the subsequent argument and so we get that the first iterated difference
of $v_n/n$ is positive
for all $n\ge 3$ and the second iterated difference of $v_n/n$ is negative for $n\ge 4$. The remaining small values of $n$ are checked by hand.

Regarding the Euler numbers, we use the inequality
 \begin{equation}
 \label{eq:Euler}
\frac{4^{2n+1} (2n)!}{\pi^{2n+1}}> |E_{2n}|>\frac{4^{2n+1} (2n)!}{\pi^{2n+1}} \left(\frac{1}{1+3^{-2n-1}}\right).
\end{equation}
Since
$$
1>\frac{1}{1+3^{-2n-1}}>1-\frac{1}{3^{2n+1}},
$$
we can write
$$
|E_{2n}|= 16^{n} \left(\frac{2 (2n)!}{(2\pi)^{2n}}\right) \left(\frac{2}{\pi}\right) (1+\eta_n),
$$
where
\begin{equation}
\label{eq:boundonetan}
0<|\eta_n|<\frac{1}{3^{2n+1}}<\frac{3}{2^{2n}}.
\end{equation}
Writing $|E_{2n}|=\exp v_n$ and following along calculation \eqref{eq:10}, we get that
\begin{eqnarray}
\label{eq:Euler1}
\frac{v_n}{n} & = & 4\log 2+\frac{1}{n} \log\left(\frac{2(2n)!}{(2\pi)^{2n}}\right) +\frac{\log (2/\pi)}{n}+\frac{\log(1+\eta_n)}{n}\nonumber\\
& = & 2\log n+c_1+\frac{\log n}{2n}+\frac{\log(64/\pi)}{2n}+\frac{\theta_{2n}}{n}+\frac{\log (1+\eta_n)}{n}.
\end{eqnarray}
Since now $\eta_n$ is negative, instead of \eqref{eq:logg} we need to use
$$
|\log(1-x)|\le 2|x|\quad {\text{\rm for}}\quad x\in [0,1/2]
$$ 
with $x=-\eta_n$ and $n\ge 2$. Comparing the last row of \eqref{eq:10} with the last row of \eqref{eq:Euler1}, we see that the only differences are in the value of $c$ and the positive constant $\log(16 \pi)$ has been replaced by the smaller positive constant $\log(64/\pi)$. So, in \eqref{eq:Delta1Ber} and \eqref{eq:Delta2}, aside from replacing $\log(16\pi)$ by $\log(64/\pi)$, also the terms $6/2^{2n}n$ and $12/2^{2n}n$ need to be replaced by their doubles $12/2^{2n}n$ and $24/2^{2n}n$, respectively. As in the case of the Tangent numbers,  such changes do not induce any significant change and so we get that the first iterated difference of $v_n/n$ is positive
for all $n\ge 3$ and the second iterated difference of $v_n/n$ is negative for $n\ge 4$, and the remaining values are checked by hand.
\end{proof}

\section{Ap\'ery, Delannoy and Franel numbers}

Let ${\bf r}=(r_0,r_1,\ldots,r_m)$ be fixed nonnegative integers and put
\begin{equation}
\label{eq:Sn}
S^{({\bf r})}(n)=\sum_{k=0}^n \binom{n}{k}^{r_0} \binom{n+k}{k}^{r_1}\cdots \binom{n+km}{k}^{r_m}\qquad {\text{\rm for}}\qquad n\ge 0.
\end{equation}
In what follows, we put $r=r_0+\cdots+r_m$. We assume that $r_0>0$. When ${\bf r}=(r)$ for some positive integer $r$, we get that
\begin{equation}
\label{eq:ris2}
S^{(r)}(n)=\sum_{k=0}^n \binom{n}{k}^r=b_n^{(r)}\quad {\text{\rm for~all}}\quad n\ge 0,
\end{equation}
where $b^{(1)}_n=2^n$, $b^{(2)}_n=\binom{2n}{n}$ is the middle binomial coefficient, and $b^{(3)}_n$ is the Franel number. When ${\bf r}=(1,1)$, we get that
\begin{equation}
\label{eqris1,1}
S^{(1,1)}(n)=\sum_{k=0}^n \binom{n}{k}^2\binom{n+k}{k}^2=d_n\quad {\text{\rm for~all}}\quad n\ge 0,
\end{equation}
is the central Delannoy number. When  ${\bf r}=(2,2)$, we get that
\begin{equation}
\label{eq:ris2,2}
S^{(2,2)}(n)=\sum_{k=0}^n \binom{n}{k}^2\binom{n+k}{k}^2=A_n\quad {\text{\rm for~all}}\quad n\ge 0,
\end{equation}
where $A_n$ is the $n$th Ap\'ery number. The next result answers in the affirmative the three Conjectures 3.8--3.10 from~\cite{Sun}.

\begin{theorem}
\label{thm:ADF}
For each ${\bf r}$ such that $r>1$, there exists $n_r$ such that the sequence $(S^{({\bf r})}_{n+1})^{1/(n+1)}/(S^{({\bf r})}_n)^{1/n}$ is strictly decreasing for $n\ge n_{\bf r}$.
\end{theorem}

\begin{proof}
We start with McIntosh's asymptotic formula for $S^{({\bf r})}(n)$ (see \cite{Mc}).

\begin{lemma}
\label{lem:Mc}
For each nonnegative integer $p$,
\begin{equation}
\label{eq:AsforSn}
S^{({\bf r})}(n)=\frac{\mu^{n+1/2}}{\sqrt{\nu(2\pi \lambda n)^{r-1}}} \left(1+\sum_{k=1}^p \frac{R_k}{n^k}+O\left(\frac{1}{n^{p+1}}\right)\right),
\end{equation}
where $0<\lambda<1$ is defined by
\begin{eqnarray*}
1 & = & \prod_{j=0}^m \left(\frac{(1+ j\lambda)^j}{\lambda(1+(j-1)\lambda)^{j-1}}\right)^{r_j},\\
\mu & = & \prod_{j=0}^m \left(\frac{1+j\lambda}{1+(j-1)\lambda}\right)^{r_j},\\
\nu & = & \sum_{j=0}^m \frac{r_j}{(1+(j-1)\lambda)(1+j\lambda)},
\end{eqnarray*}
and each $R_k$ is a rational function of the exponents $r_0,r_1,\ldots,r_m$ and $\lambda$.
\end{lemma}

Put $f(n)$ for the function such that
$$
S^{({\bf r})}_n=\frac{\mu^{n+1/2}}{\sqrt{\nu(2\pi \lambda n)^{r-1}}} f(n).
$$
Put $S^{({\bf r})}_n=\exp v_n$. Then
$$
\frac{v_n}{n}=\log \mu+\frac{c}{n}-\frac{(r-1)\log n}{2n}+\frac{\log(f(n))}{n},
$$
where $c=\mu^{1/2} \nu^{-1/2} (2\pi \lambda)^{(1-r)/2}.$ Thus,
\begin{eqnarray*}
\Delta^{(2)} \left(\frac{v_n}{n}\right) & = & \frac{2c}{n(n+1)(n+2)}\\
& - & \frac{(r-1)}{2} \left(\frac{\log(n+2)}{n+2}-\frac{2\log(n+1)}{n+1}+\frac{\log n}{n}\right)\\
& + & \left(\frac{\log(f(n+2))}{n+2}-\frac{2\log (f(n+1))}{n+1}+\frac{\log (f(n))}{n}\right).
\end{eqnarray*}
The argument from the proof of Theorem \ref{thm:Bern} shows that
\begin{eqnarray*}
\frac{\log(n+2)}{n+2}-\frac{2\log(n+1)}{n+1}+\frac{\log n}{n} & = & \frac{2\log n}{n(n+1)(n+2)}+O\left(\frac{1}{n^3}\right)\\
& = & \frac{2\log n}{n^3}+
O\left(\frac{1}{n^3}\right).
\end{eqnarray*}

Next, write
\begin{equation}
\label{eq:f}
f(x)=1+\frac{R}{x}+O\left(\frac{1}{x^2}\right)
\end{equation}
for some rational $R$ as in Lemma~\ref{lem:Mc}. For simplicity, put
$$
g(x)=1+\frac{R}{x}.
$$
Thus,
\begin{equation}
\label{eq:20}
\frac{\log f(x)}{x}=\frac{\log (g(x))}{x}+O\left(\frac{1}{x^3}\right).
\end{equation}
Furthermore, by the Intermediate Value Theorem, we have
\begin{eqnarray*}
\left|\Delta^{(1)}\left( \frac{\log g(m))}{m}\right)\right| & = & \left|\frac{d}{dx} \left(\frac{\log g(x)}{x}\right)\Big|_{x=\zeta\in [m,m+1]}\right|\\
& = &
\left|\frac{\zeta g'(\zeta)/g(\zeta)+\log(g(\zeta))}{\zeta^2}\right|\\
& = & O\left(\frac{1}{m^3}\right)
\end{eqnarray*}
for large enough positive integers $m$, simply by differentiating the form~\eqref{eq:20}, and using the interval for
$\zeta$.

Further, this shows that
$$
\frac{\log(f(n+2))}{n+2}-\frac{2\log (f(n+1))}{n+1}+\frac{\log (f(n))}{n}=O\left(\frac{1}{n^3}\right).
$$
Hence,
$$
\Delta^{(2)} \left(\frac{v_n}{n}\right)=(r-1)\frac{\log n}{n^3} +O\left(\frac{1}{n^3}\right),
$$
and the above expression is positive when $r>1$ for $n>n_{\bf r}$, which is what we wanted to prove.
\end{proof}

\section{Motzkin numbers, Schr\"oder numbers and Trinomial coefficients}

The $n$th Motzkin number is
$$
M_n=\sum_{k=0}^{\lfloor{n/2}\rfloor} \binom{n}{2k}\binom{2k}{k}\frac{1}{k+1}
$$
and counts the number of lattice paths from $(0,0)$ to $(n,0)$ which never dip below the line $y=0$ and which are made up only of steps
$(1,0),~(1,1)$ and $(1,-1)$.

The $n$th Schr\"oder number is
$$
S_n=\sum_{k=0}^n \binom{n}{k}\binom{n+k}{k}\frac{1}{k+1}
$$
and counts the number of lattice paths form $(0,0)$ to $(n,n)$ with steps $(1,0),~(0,1)$ and $(1,1)$ that never rise above the line $y=x$.

The  $n$th trinomial coefficient ${\rm Tr}_n$ is the coefficient of $x^n$ in the expansion of $(x^2+x+1)^n$. Its formula is
$$
{\rm Tr}_n=\sum_{k=0}^n \binom{n}{k}\binom{n-k}{k}.
$$
The following result gives a partial affirmative answer (up to the values of $n_0$) to Conjectures 3.6, 3.7 and 3.11 from~\cite{Sun}.

\begin{theorem}
Each of the sequences $(M_n^{1/n})_{n\ge n_0}$, $(S_n^{1/n})_{n\ge n_0}$ and $({\rm Tr}_n^{1/n})_{n\ge n_0}$
is strictly increasing while  each of   $(M_{n+1}^{1/(n+1)}/M_n^{1/n})_{n\ge n_0}$,
$(S_{n+1}^{1/(n+1)}/S_n^{1/n})_{n\ge n_0}$ and $({\rm Tr}_{n+1}^{1/(n+1)}/{\rm Tr}_n^{1/n})_{n\ge n_0}$ is strictly decreasing.
\end{theorem}

\begin{proof}
It is similar to the proof of Theorem \ref{thm:ADF} and it is based on the existence of analogues of asymptotic expansions for $M_n,~S_n$ and ${\rm Tr}_n$ to the one of Lemma \ref{lem:Mc} (in fact, as we have seen in the proof of Theorem \ref{thm:ADF}, the existence of an expansion with the first two terms, as in \eqref{eq:f}, suffices).  For example,
\begin{equation}
\label{eq:assMn}
M_n=\sqrt{\frac{3}{4\pi n^3}} 3^n\left(1-\frac{15}{16 n}+\frac{505}{512 n^2}-\frac{8085}{8192 n^3}+\frac{505659}{524288 n^4}+O\left(\frac{1}{n^5}\right)\right)
\end{equation}
(see Example VI.3 on page 396 in \cite{FlaSed}),
\begin{equation}
\label{eq:assSn}
S_n=\sqrt{\frac{4+3{\sqrt{2}}}{4\pi n^3}} (3+2{\sqrt{2}})^n \left(1-\frac{24+9{\sqrt{2}}}{32n}+\frac{665+360{\sqrt{2}}}{1024 n^2}+
O\left(\frac{1}{n^3}\right)\right)
\end{equation}
(see \cite{Wag1}), and
\begin{equation}
\label{eq:assTrn}
{\rm Tr}_n=\sqrt{\frac{1+{\sqrt{2}}}{4\pi n}} (1+{\sqrt{2}})^n \left(1-\frac{3}{16n}+O\left(\frac{1}{n^2}\right)\right)
\end{equation}
(see \cite{Wag2}). We give no further details.
\end{proof}

\section{Acknowledgements}

We thank Stephan Wagner for useful correspondence and for providing us the asymptotics \eqref{eq:assSn} and Zhi-Wei Sun for providing some references. This paper was written during a visit of P.~S. to the Centro de Ciencias Matem\'aticas de la UNAM in Morelia in August 2012. During the preparation of this paper,  F.~L. was supported in part by Project
PAPIIT IN104512 (UNAM),  VSP N62909-12-1-4046 (Department of the US Navy, ONR)  and a Marcos Moshinsky fellowship.


\begin{thebibliography}{99}

\bibitem{AS} M. Abramowitz, I. A. Stegun, (Eds.), ``Bernoulli and Euler Polynomials and the Euler-MacLaurin Formula'', \S23.1 in {\em Handbook of Mathematical Functions with Formulas, Graphs, and Mathematical Tables}, 9th printing, New York: Dover, 1972.

\bibitem{BBD}
J. M. Borwein, P. B. Borwein, K. Dilcher, ``Pi, Euler Numbers, and Asymptotic Expansions'',
  {\em American Math. Monthly} {\bf 96:8} (1989), 681--687.

\bibitem{FlaSed} P. Flajolet and R. Sedgewick, ``Analytic combinatorics'', Cambridge U. Press, 2009.

\bibitem{HSW12} Q.--H. Hou, Z.--W. Sun, H. We,
{\em On monotonicity of some combinatorial sequences}, {\it Preprint\/}, http://arxiv.org/pdf/1208.3903.

\bibitem{Mc} R.~McIntosh, ``An asymptotic formula for binomial sums'', {\it J. Number Theory\/} {\bf 58} (1996), 158--172.

\bibitem{Sun} Z.--W. Sun, ``Conjectures involving combinatorial sequences'', {\it Preprint\/}, http://arxiv.org/pdf/1208.2683.

\bibitem{Wag1} S.~Wagner, Personal communication, August 21, 2012.

\bibitem{Wag2} S.~Wagner, ``Asymptotics of generalized trinomial coefficients'', {\it Preprint\/}, arXiv: 1205.5402v3.


\end{thebibliography}
\end{document}